\documentclass[12pt, reqno]{amsproc}
\usepackage{graphicx, enumerate}
\usepackage[margin=1in]{geometry}
\usepackage{amsmath}
\usepackage{amsfonts}
\usepackage{amssymb}
\usepackage{amsfonts,amsmath,amssymb}
\usepackage{amsthm}
\newtheorem{lemma}{Lemma}
\newtheorem{theorem}{Theorem}

\newtheorem{thm}{Theorem}

\newtheorem{example}{Example}
\newtheorem{remark}{Remark}
\title[Study of Growth of Certain Second Order Linear Differential Equations]{Study of Growth of Certain Second Order Linear Differential Equations}
\author[Naveen Mehra, Garima Pant and S. K. Chanyal]{Naveen Mehra, Garima Pant and S. K. Chanyal}
\address{Naveen Mehra, Department of Mathematics, Kumaun University, D.S.B. Campus, Nainital-263001, Uttarakhand, India.}
\email{naveenmehra00@gmail.com}
\address{Garima Pant, Department of Mathematics, University of Delhi, Delhi-110007, India.}
\email{garimapant.m@gmail.com}
\address{S. K. Chanyal, Department of Mathematics, Kumaun Univesity, D.S.B. Campus, Nainital-263001, Uttarakhand, India.}
\email{skchanyal.math@gmail.com}
\begin{document}
	\maketitle
	\begin{abstract}
In this article, we study about the solutions of second order linear differential equations by considering several conditions on the coefficients of homogenous linear differential equation and its associated non-homogenous linear differential equation.
	\end{abstract}
	\section{Introduction}
Consider homogenous linear complex differential equation
\begin{equation}\label{2order}
		f''+A(z)f'+B(z)f=0,
	\end{equation}
	where $A(z)$ and $B(z)\not\equiv 0$ are entire functions. All the solutions of equation \eqref{2order} are of finite order if and only if the coefficient $A_(z)$ and $B(z)$ are polynomials (see \cite{wittich1966}). It is a natural question to arise that what happens when atleast one of the coefficient is transcendental entire function. M. Frei \cite{frei1961losungen} answered to this question. He proved that atmost all non-trivial solution of equation \eqref{2order} are of infinite order.\\

The main aim of this work is to find conditions on entire coefficients $A(z)$ and $B(z),$ so that all non-trivial solutions of equation \eqref{2order} are of infinite order. Many researchers studied this problem earlier. Gundersen \cite{gundersen1988finite} proved that if $\rho(A)<\rho(B)$, then all non-trivial solutions of equation \eqref{2order} are of infinite order. It is clear that if $A(z)$ is a polynomial and $B(z)$ is a transcendental entire function, then all non-trivial solutions are of infinite order. But, the case $\rho(A)\geq\rho(B)$ was unexplored until the paper by Ozawa \cite{ozawa1980solution}. After Ozawa's paper, many other researchers studied the same case partially. 
The following result is the collection of those results.
\begin{theorem}
All non-trivial solution of equation \eqref{2order} are of infinite order if the coefficients $A(z)$ and $B(z)$ satisfy any of the following conditions
\begin{enumerate}[(a)]
\item\cite{gundersen1988finite} $\rho(A)<\rho(B)$;
\item\cite{hellerstein1991growth} $\rho(B) < \rho(A) \leq \frac{1}{2}$;
\item\cite{gundersen1988finite} $A(z)$ is a transcendental entire function with $\rho(A) = 0$ and $B(z)$ is a polynomial;
\item\cite{gundersen1988finite} $A(z)$ is a polynomial and $B(z)$ is a transcendental entire function.
\end{enumerate}
\end{theorem}
\begin{example}
\begin{enumerate}[(i)]
\item $f''+e^zf'+(e^z-1)f=0$ has solution $f(z)=e^z.$
\item $f''+(\sin^2z-2\tan z)f'-\tan zf=0$ has solution $f(z)=\tan z$.
\end{enumerate}
\end{example}
It can be observed from the above examples that differential equations have finite order solutions  when $\rho(A)=\rho(B)$ or $\rho(A)>\rho(B)$ and $\rho(A)>1/2.$ There arises a question that under what circumstances equation \eqref{2order} possesses all non-trivial solutions of infinite order with these conditions. In the next section we partially answer this question.\\ 

The corresponding non-homogenous second order linear differential equation is
\begin{equation}\label{2ordernonhomo}
	f''+A(z)f'+B(z)f=H(z),
\end{equation}
where $A(z)$, $B(z)$ and $H(z)$ are entire functions.
A non-homogenous linear differential equation can always be reduced back to homogenous linear differential equation. So, the basic results are similar to as in homogenous case. If all the coefficients and $H(z)$ are entire functions, then all the solutions of equation \eqref{2ordernonhomo} are also entire functions (see \cite{shi1991two}). If all the coefficients are polynomials and $H(z)\neq 0$ have finite order of growth, then all solutions of equation \eqref{2ordernonhomo} are of finite order of growth (see \cite[Lemma 2]{gao1989complex}). Therefore, if atleast one of the coefficient is a transcendental entire function then atmost all solutions are of infinite order. Let $\rho$ be the minimal order of solutions of equation \eqref{2order}, then it is completely elementary that there may exist atmost one solution of order less than $\rho$ of equation \eqref{2ordernonhomo} (see \cite{laine1990note}). Thus, if all non-trivial solutions of equation \eqref{2order} are of infinite order, there may exist finite order solution of equation \eqref{2ordernonhomo}.   We illustrate this fact by the following examples.
\begin{example}\label{nordernonhomoeg1}
	The equation $$f''+z f' + e^z f=e^{-z}(1-z) +1$$ has a finite order solution, that is, $f(z)=e^{-z},$ whereas the associated \linebreak homogenous equation has all non-trivial solutions of infinite order.
\end{example}

\begin{example}\label{nordernonhomoeg2}
	Let $b(z)$ be a finite order entire function and has multiply connected Fatou component. Then, the equation $$f''-e^zf'+b(z)f=0,$$ has all non-trivial solutions of infinite order (see \cite[Theorem B]{mehra2022growth}). But the \linebreak associated non-homogenous equation $$f''-e^zf'+b(z)f=e^{-z}(1+b(z))+1$$ has finite order solution $f(z)=e^{-z}.$
\end{example}
	\section{Results}
	\subsection{Second Order Homogenous Linear Differential Eqaution}

G. Zhang \cite{zhang2018infinite} in his paper investigated the solutions of equation \eqref{2order}, when $A(z)=e^{P(z)}$ such that $P(z)$ and $B(z)$ are poynomials of degree $m$ and $n$ respectively.

\begin{theorem}\cite{zhang2018infinite}\label{zhang}
	Suppose $A(z) = e^{p(z)}$, where $p(z)$ is a polynomial with degree $n\geq 2$ and $B(z) = Q(z)$ is also a nonconstant polynomial with degree $m$. If $m + 2 > 2n$ and $n\nmid m + 2$. Then, every solution $f (\not\equiv 0)$ of equation 
	$$f''+ e^{p(z)}f'+ Q(z)f = 0 $$
	is of infinite order.\end{theorem}

In our first main result, we replace $A(z)=e^{P(z)}$ by $A(z)=h(z)e^{P(z)}$ in Theorem \ref{zhang}, where $\rho(h)<n$. 
\begin{thm}\label{B(z)polynomial}
	Suppose $A(z) = h(z)e^{p(z)}$, where $p(z)$ is a polynomial with degree $n\geq 2$ and $B(z) = Q(z)$ is also a nonconstant polynomial with degree $m$. If $m + 2 > 2n$ and $n\nmid m + 2$. Then, every solution $f (\not\equiv 0)$ of equation 
	\begin{equation}\label{Q(z)poly}
		f''+ h(z)e^{p(z)}f'+ Q(z)f = 0
	\end{equation} 
	is of infinite order.
\end{thm}
Following Lemma is due to Bank, et al.\cite{banklainelangley} that gives an estimate for an entire function with an integral order and the asymptotic properties on most rays of the function
$h(z)e^{P(z)}$. 
\begin{lemma}\cite{banklainelangley}\label{lcritical}
	Let $A(z) = h(z)e^{P(z)}$ be an entire function with $\lambda(A) < \rho(A) = n$, where $P(z)$ is a polynomial of degree $n$. Then, for
	every $\epsilon > 0,$ there exists $E \subset [0, 2\pi)$ of linear measure zero satisfying
	\begin{enumerate}[(i)]
		\item for $\theta \in [0, 2\pi)\setminus E$ with $\delta(P, \theta) > 0$, there exists $R > 1$ such that
		$$exp ((1 - \epsilon)\delta(P, \theta)r^n) \leq |A(re^{\iota\theta})|$$
		for $r > R$;
		\item for $\theta \in [0,2\pi)\setminus E$ with $\delta(P, \theta) < 0$, there exists $R > 1$ such that
		$$|A(re^{\iota\theta})| \leq exp ((1 - \epsilon)\delta(P, \theta)r^n)$$ 
		for $r > R$.
	\end{enumerate}
\end{lemma}
The following Lemma is given by Langley\cite{langley1986complex}.

\begin{lemma}\cite{langley1986complex}\label{Q(z)}
	Let $S$ be the strip $$z = x + \iota y,\ \ x\geq x_0,\ \  |y|\leq 4.$$
	Suppose that in $S$
	$$Q(z) = a_nz^n + O(|z|^{n-2}),$$
	where $n$ is positive integer and $a_n > 0.$ Then, there exists a path $\Gamma$ tending to $\infty$ in $S$ such that all solutions of $$y''+ Q(z)y = 0$$ tend to zero on $\Gamma.$
\end{lemma}
The following Lemma gives the logarithmic estimate of a meromorphic function outside  an $R$-set.
\begin{lemma}\cite{laine1993nevanlinna}\label{f'(z)f(z)}
	Let $f$ be a meromorphic function of finite order. Then, there exists $N = N(f ) >
	0$ such that
	$$\left|\frac{f'(z)}{f(z)}\right|=O(r^N)$$
	holds outside an $R$-set.
\end{lemma}
The growth estimate in the following Lemma deduced in \cite{banklainelangley} from Herold Comparison Theorem \cite{herold1972vergleichssatz}.
\begin{lemma}\cite{langley1986complex}\label{y''+ A(z)y = 0}
	Suppose that $A(z)$ is an analytic in a sector containing the ray $\Gamma : re^{\iota\theta}$ and that as $r\to\infty$, $A(re^{\iota\theta}) = O(r^n)$ for some $n\geq 0.$ Then, all solutions of $y''+ A(z)y = 0$ satisfy
	$$\log^+ |y(re^{\iota\theta})| = O(r^{\frac{(n+2)}{2}})$$
	on $\Gamma$.
\end{lemma}
\begin{remark}\label{rem}
	If $f(z)\to a$ as $z\to\infty$ along a straight line, $f(z)\to b$ as $z\to\infty$ along another straight line and $f(z)$ is analytic and bounded in the angle between, then $a = b$ and $f(z)\to a$ uniformly in the angle. The straight lines may be replaced by the curves approaching $\infty$.
\end{remark}
\begin{proof}[\underline{Proof of Theorem \rm{\ref{B(z)polynomial}}}]
	We assume that \eqref{Q(z)poly} has a solution $f(z)$ with finite order. Set
	\begin{equation}\label{ftoy}
		f = y \exp\{-\frac{1}{2}\int_0^z  h(z)e^{p(z)}dz\}.
	\end{equation}
	Equation \eqref{Q(z)poly} can be transformed into
	\begin{equation}\label{transform}
		y'' +\left(Q(z)-\frac{1}{4} (he^{p(z)})^2-\frac{1}{2}h'(z)e^{p(z)}-\frac{1}{2}h(z)p'(z)e^{p(z)}\right)y = 0.
	\end{equation}
	By a translation, we may assume that
	$$Q(z) = a_mz^m + a_{m-2}z^{m-2} +\cdots, \ \  m > 2.$$
	We define the critical ray for $Q(z)$ as those ray $re^{\iota\theta_j}$ for which
	$$\theta_j=\frac{-\arg a_m+2j\pi}{m+2},$$
	where $j = 0, 1, 2, \ldots, m + 1$ and note that the substitution $z = xe^{\iota\theta_j}$ transforms equation \eqref{transform} into
	\begin{equation}\label{wrtx}
		\frac{d^2y}{dx^x} + (Q_1(x) + P_1(x))y = 0,
	\end{equation}
	where
	$$Q_1(x) =\alpha_1 x^m + O(x^{m-2}), \alpha_1 > 0$$
	and
	$$P_1(x) = -\frac{1}{4} (he^{p(xe^{\iota\theta_j})})^2-\frac{1}{2}h'(xe^{\iota\theta_j})e^{p(xe^{\iota\theta_j})}-\frac{1}{2}h(xe^{\iota\theta_j})p'(xe^{\iota\theta_j})e^{p(xe^{\iota\theta_j})}.$$ 
	For the polynomial $p(z)$ with degree $n$, set $p(z) = (\alpha + \iota\beta) z^n + p_{n-1}(z)$ with $\alpha , \  \beta$ real, and denote $\delta(p, \theta) = \alpha\cos n\theta - \beta\sin n\theta .$ The rays
	$$\arg z = \theta_k = \frac{arc\tan\frac{\alpha}{\beta} + k\pi}{n},\\ k = 0, 1, 2, \ldots, 2n-1$$
	satisfying $\delta(p, \theta_k) = 0$ can split the complex domain into $2n$ equal angular domains. Without loss of generality, denote these angle domains as
	$$\omega^+=\{re^{\iota\theta}:0<r<+\infty ,\frac{2i\pi}{n}<\theta<\frac{(2i+1)}{n} \},$$
	$$\omega^-=\{re^{\iota\theta}:0<r<+\infty ,\frac{(2i+1)}{n}<\theta<\frac{2(i+1)}{n} \},$$
	$\iota = 0, 1,\ldots , n - 1,$ where $\delta(p, \theta) > 0$ on $\omega^+$ and $\delta(p, \theta) < 0$ on $\omega^-$. By Lemma \ref{lcritical}, we obtain 
	\begin{align*}
		|P_1(x)|\leq & |(h(xe^{\iota\theta_j})e^{p(xe^{\iota\theta_j})})^2|+|h'(xe^{\iota\theta_j})e^{p(xe^{\iota\theta_j})}|+|h(xe^{\iota\theta_j})e^{p(xe^{\iota\theta_j})}p'(xe^{\iota\theta_j})| \\
		&\leq \exp\{\delta(P,\theta)x^n\}+\exp\{\frac{1}{2}\delta(P,\theta)x^n\}+\exp\{\frac{1}{2}\delta(P,\theta)x^n\}O(x^{n-1})\to0
	\end{align*}
	for $xe^{\iota\theta_j}\in\omega^-$ as $x\to\infty$, then by Lemma \ref{Q(z)} and \eqref{wrtx}, for any critical line $\arg z = \theta_j$ lying in $\omega^-$ there exists a
	path $\Gamma_{\theta_j}$ tending to $\infty$, such that $\arg z\to\theta_j$ on $\Gamma_{\theta_j}$ while $y(z)\to 0$ there. Moreover, by
	\begin{align}\label{to1}
		|\exp\{-\frac{1}{2}\int_0^zh(z)e^{p(z)}\}| & \leq\exp\{\frac{1}{2}\left|\int_0^zh(z)e^{p(z)}\right|\}\\ \nonumber
		& \leq\exp\{\frac{1}{2}r\exp\{\delta(p, \theta )r^n\}\}\to 1
	\end{align}
	for $z >\omega^-$ as $r\to\infty$, together with \eqref{ftoy} we have $f (z)\to 0$ along $\Gamma_{\theta_j}$ tending to $\infty$.
	Setting $V =\frac{f'}{f}$ , equation \eqref{Q(z)poly} can be written as
	$$V'+V^2 + h(z)e^{p(z)}V + Q(z) = 0.$$
	By Lemma \ref{f'(z)f(z)}, we have
	$$|V'|+|V|^2 = O(|z|^N)$$
	outside an $R$-set $U$, where $N$ is a positive constant. Moreover, if $z = re^{\iota\phi}\in\omega^+$ is such that the ray $\arg z =\phi$ meets only finitely many discs of $U$ we see that $V = o(|z|^{-2})$ as $z$ tends to $\infty$ on this ray and hence $f$ tends to a finite, nonzero limit. Applying this reasoning to a set of $\phi$ outside a set of $0$ measure we deduce by the Phragm\'en-Lindel\"{o}f principle that without loss of generality, for any small enough given positive $\epsilon$,
	\begin{equation}\label{fto1}
		f (re^{\iota\theta})\to 1,
	\end{equation}
	as $r\to\infty$ with
	$$z=re^{\iota\theta}\in\omega^+_\epsilon =\{z=re^{\iota\theta}: 0<r<\infty , \frac{2i\pi}{n}+\epsilon <\theta <\frac{(2i+1)\pi}{n}-\epsilon\}.$$
	For any $z = re^{\iota\theta}\in\omega^-$, we have that $\delta(p, \theta) < 0$, and by Lemma \ref{lcritical}, we have
	\begin{align}\label{abc}
		|Q(z)-\frac{1}{4} (he^{p(z)})^2-\frac{1}{2}h'(z)e^{p(z)} & -\frac{1}{2}h(z)p'(z)e^{p(z)}| 
		\leq |Q(z)|+|(he^{p(z)})^2| \nonumber \\
		& +|h'(z)e^{p(z)}|+|h(z)p'(z)e^{p(z)}| \nonumber \\ 
		&\leq O(r^m)+\exp\{\delta(P,\theta)x^n\}+\exp\{\frac{1}{2}\delta(P,\theta)x^n\} \nonumber \\ 
		&+\exp\{\frac{1}{2}\delta(P,\theta)x^n\}O(x^{n-1}) \nonumber \\ 
		& \leq O(r^m)
	\end{align}
	for sufficiently large $r$. Applying, Lemma \ref{y''+ A(z)y = 0} to \eqref{transform} and together with \eqref{abc}, $y(z)$ satisfies
	$$\log^+ |y(re^{\iota\theta}| = O(r^{\frac{m+2}{2}})$$
	as $r\to\infty$ for any $z = re^{\iota\theta}\in\omega^-.$ From \eqref{transform} and \eqref{to1}, we have
	\begin{equation}\label{O(r)}
		\log^+|f(re^{\iota\theta}| = O(r^\frac{m+2}{2})
	\end{equation}
	as $r\to\infty$ for any $z = re^{\iota\theta}\in\omega^-.$
	On the rays $\arg z = \theta_k$ such that $\delta(p, \theta_k) = 0,$ we have $|e^{p(z)}| =|e^{p_{n-1}(z)}|.$ Consider the two cases
	$\delta(p_{n-1}, \theta_k) > 0$ or \linebreak $\delta(p_{n-1}, \theta_k) < 0$, by the same method above, we get $f (z)\to 1$ or $\log^+ |f (z)| = O(r^\frac{m+2}{2}),$
	respectively, on the ray $\arg z = \theta_k.$ If $\delta(p_{n-1}, \theta_k) = 0$ also, repeating these arguments again. Finally, we deduce
	that either $f (z)\to 1$ or $\log ^+|f(z)| = O(r^\frac{m+2}{2})$ on the rays $\arg z = \theta_k ,\ \ k = 0, 1,\ldots , 2n-1.$ Thus, \eqref{ftoy}, \eqref{fto1}, \eqref{O(r)}
	and the fact $\epsilon$ is arbitrary imply that, by the Phragm\'en-Lindel\"{o}f principle,
	\begin{equation}\label{orderf<m+22}
		\rho(f)\leq\frac{m+2}{2}.
	\end{equation}
	We claim that $\frac{(2i+1)}{n}, \ \ (i = 0, 1,\ldots, n-1)$ are critical rays for $Q(z)$. Otherwise, there exists a critical $\theta_j$ for $Q(z)$ in 
	$$\frac{2i+1\pi}{n}<\theta_j<\frac{2(i+1)\pi}{n}+\frac{2\pi}{m+2} \ \ (i= 0, 1, \ldots , n-1) $$
	because $m+2 > 2n.$ This implies the existence of an unbounded domain of angular measure at most $\frac{2\pi}{m+2}+\epsilon$,
	bounded by a path on which $f (z)\to 0$ and a ray on which $f (z)\to 1.$ By Remark \ref{rem} implies that $\rho(f ) >\frac{m+2}{2},$ contradicting \eqref{orderf<m+22}. Then, there exists a positive integer $k$ satisfying $\frac{2\pi}{n} = k\frac{2}{m+2},$ that is, $m + 2 = kn$, which contradicts $n \nmid m + 2.$ Thus, we complete the proof.
\end{proof}	
Theorem \ref{mainth2} is motivated by Theorem \ref{kumar2021growth.} given by Kumar and Saini\cite{kumar2021growth}. They considered, $A(z)$ has Fabry gaps and $\rho (B) < \rho (A)$. We changed the conditions on $A(z)$ to have multiply connected Fatou component.

\begin{theorem}\cite{kumar2021growth}\label{kumar2021growth.}
	Let $A(z)$ and $B(z)$ be an entire functions such that $\rho (B) < \rho (A)$ and
	$A(z)$ has Fabry gaps. Then, $\rho(f) = \infty$ and $\rho_2(f) = \rho (A),$
	where $f$ is a non-trivial solution of equation \eqref{2order}.
	\end{theorem}

\begin{thm}\label{mainth2}
	Let $A(z)$ be a transcendental entire function with a multiply-connected Fatou component and $B(z)$ be an entire function satisfying $\rho(B)<\rho(A)$. Then, every non-trivial solution of equation \eqref{2order} is of infinite order. Moreover,
	$$\rho_{2}(f)=\rho(A).$$
\end{thm}
Lemma \ref{lgundersen} is given by Gundersen\cite{gundersen}. He generalized the estimates of logarithmic derivatives of transcendental meromorphic function of finite order. 	
\begin{lemma}\cite{gundersen}\label{lgundersen}
Let $f$ be a transcendental meromorphic function with
finite order and $(k, j)$ be a finite pair of integers that satisfies $k > j \geq 0$ and let $\epsilon>0$ be a given constant. Then following statements holds:
\begin{enumerate}[(a)]
		\item  there exists a set $E_1 \subset [0, 2\pi]$ with linear measure zero such that for $\theta \in [0, 2\pi) \setminus E_1$ there exists $R(\theta) >0$ such that 
		$$\left|\frac{f^{(k)}(z)}{f^{(j)}(z)}\right|\leq|z|^{(k-j)(\rho(f)-1+\epsilon)}$$
		for all $k, j$; $| z |> R(\theta)$ and $arg z = \theta$
		\item there exists a set $E_2 \subset (1,\infty)$ with finite logarithmic measure  such that for all $|z|\not\in E_2\cup[0,1]$ such that inequality in (a) holds  for all $k, j$ and $| z |\geq R(\theta)$.
		\item there exists a set $E_3 \subset [0,\infty)$ with finite linear measure  such that for all $|z|\not\in E_3$ such that 
		$$\left|\frac{f^{(k)}(z)}{f^{(j)}(z)}\right|\leq|z|^{(k-j)(\rho(f)+\epsilon)}$$
		holds  for all $k, j$.
	\end{enumerate}
\end{lemma}

Recently, Pant and Saini \cite{pant2021infinite} proved the following result for an entire function. 

\begin{lemma}\cite{pant2021infinite}\label{lpantsaini}
	Suppose $f$ is a transcendental entire funtion. Then, there
	exists a set $F\subset(0,\infty)$ with finite logarithmic measure such that for all $z$ satisfying $|z| = r\in F$ and $|f(z)| = M(r, f)$ we have
	$$
	\left|\frac{f(z)}{f^{(m)}(z)}\right| \leq 2r^m,
	$$
	for all $m\in N.$
\end{lemma}
\begin{lemma}\cite{zheng}\label{lzheng}
	Suppose $f$ is a transcendental meromorphic function having atmost finite poles. If $J(f)$ has only bounded components, then for any complex number, there exists a constant $0<\beta<1$ and two sequences of positive numbers $\{r_n\}$ and $\{R_n\}$ with $r_n\to\infty$ and $R_n/r_n\to\infty(n\to\infty)$ such that
	$$M(r,f)^{\beta}\leq L(r,f)\quad\mbox{for}\quad r\in H,$$
	where $H=\cup_{n=1}^\infty\{r:r_n<r<R_n\}.$  
\end{lemma}
\begin{proof}[\underline{Proof of Theorem \rm{\ref{mainth2}}}] 
	We prove this Theorem by contradiction. Suppose $f$ is a finite order non-trivial solution of equation \eqref{2order}.
		Applying Lemma \ref{lgundersen}, there is a set $E\subset(1,\infty)$ with finite logarithmic measure such that
		\begin{equation}\label{guneqch4}
			\left| \frac{f^{''}(z)}{f^{'}(z)}\right| \leq |z|^{2\rho(f)},	
		\end{equation}
		holds for all $z$ satisfying $|z|\notin E\cup [0,1]$.\\
		Suppose that $z_{r}=re^{\iota\theta_{r}}$ be the points such that $|f(z_{r})|=M(r,f)$. Then, applying Lemma \ref{lpantsaini}, there exists a set $F\subset(0, \infty)$ with
		$m_l(F)<\infty$ such that 
		\begin{equation}\label{impeqch4}
			\frac{f(re^{\iota\theta_{r}})}{f^{m}(re^{\iota\theta_{r}})}\leq 2r^{m},
		\end{equation}
		holds for all sufficiently large $r\notin F$ and for all $m\in \mathbb{N}$.
		Applying Lemma \ref{lzheng}, we have 
		\begin{equation}\label{fatoueqch4}
			M(r,A)^{\gamma}\leq |A(re^{\iota\theta})|,
		\end{equation}
		for $0<\gamma<1$ and $r\in F_{1}=\cup_{n=1}^\infty\{r:r_n<r<R_n\}$.
		Let $\rho(B)<\beta<\rho(A)$, then the
		definition of order of growth of $B(z)$ implies that
		\begin{equation}\label{ordereqch4}
			|B(re^{\iota\theta})|\leq \exp r^{\beta},
		\end{equation}
		for all sufficiently large $r$.  
		From equations \eqref{2order}, \eqref{guneqch4}, \eqref{impeqch4}, \eqref{fatoueqch4} and \eqref{ordereqch4}, there exists a sequence $z=re^{\iota\theta}$ such that for all $r\in F_{1}\setminus (F\cup E\cup[0,1])$, we have
		\begin{align*}
			|A(re^{\iota\theta})|&\leq \left|\frac{f^{''}(re^{\iota\theta})}{f^{'}(re^{\iota\theta})}\right|+|B(re^{\iota\theta})|\left|\frac{f(re^{\iota\theta})}{f^{'}(re^{\iota\theta})}\right|\\
			\implies M(r, A)^{\gamma}&\leq r^{2\rho(f)}+2r\exp r^{\beta}\\
			&\leq 2r\exp r^{\beta}(1+o(1)).
		\end{align*}
		This gives $\rho(A)\leq \beta$, which is a contradiction. Hence, every non-trivial solution of equation \eqref{2order} is of infinite order.
	\end{proof}
In 2017, Gundersen\cite{gundersen2017research} asked a question, ``Does every non-trivial solution $f$ of equation \eqref{2order} is of infinite order, when $A(z)$ satisfies $\lambda(A)<\rho(A)$ and $B(z)$ is a non-constant polynomial?'' Long, et al.\cite{2018queslongtion} partially answered the question.
\begin{theorem}\cite{2018queslongtion}\label{2018queslongtion}  Let $A(z)=h(z)e^{P(z)}$ satisfy $\lambda(A)< \rho (A)$ and $B(z)=b_mz^m+b_{m-1}z^{m-1}+\cdots + b_0$ is a polynomial of degree $m$ such that:
	\begin{enumerate}[(a)]
		\item $m + 2 < 2n$, or
		\item $m + 2 > 2n$ and $m + 2 \neq 2kn$ for all integers $k$, or
		\item $m + 2 = 2n$ and $\frac{a_n^2}{b_m}$ is not real and negative.
	\end{enumerate}
	Then, all non-trivial solution of equation \eqref{2order} have infinite order.
\end{theorem}
 Kumar, et al.\cite{kumarsaini} motivated by their result, considered $\rho(A)>n$ and $B(z)$ to be a polynomial in Theorem \ref{Manisha}. 
\begin{theorem}\cite{kumarsaini}\label{Manisha}
	Consider a transcendental entire function $A(z)=h(z)e^{P(z)}$, where $P(z)$ is a non-constant polynomial of degree $n$ and $\rho(h) > n$. Assume that $h(z)$ is bounded away from zero and exponentially blows up in $E^+$ and $E^-$ respectively and let $B(z)$ be a polynomial.
	Then, all non-trivial solutions of the equation (\ref{2order}) are of infinite order.
	\end{theorem}
	
We are motivated by Theorem \ref{Manisha} and replace the condition of $B(z)$ to satisfy some conditions given in Theorem \ref{th2}.

\begin{thm}\label{th2} Let $A(z)$ satisfy the conditions of Theorem \ref{Manisha} and $B(z)$ be a transcendental entire function satisfying
	\begin{enumerate}
		\item[(i)] $\rho(B)<\rho(A)$ or
		\item[(ii)]$\mu(B)<\rho(A)$
	\end{enumerate}
	Then, all non-trivial solutions of the equation \eqref{2order} are of infinite order.
\end{thm}
The following Lemma yield us a lower bound for modulus of an entire function in the neighbourhood of $\theta$, where $\theta\in [0, 2\pi).$
\begin{lemma}\cite{wanglaine}\label{fromnonhomo}
	Suppose $f(z)$ is an entire function of finite order $\rho$ and $M(r,f)=|f(re^{\iota\theta_r})|$ for every $r$. Given $\zeta >0$ and $0<C(\rho, \zeta )<1,$ there exists $0<l_0<\frac{1}{2}$ and a set $S \subset (1,\infty)$ with $\underline{\log dens}(S) \geq 1 - \zeta$ such that
	$$e^{-5\pi} M(r, f)^{1-C} \leq |f(re^{\iota \theta})|,$$ 
	for all sufficiently large $r \in S$ and for all $\theta$ satisfying $|\theta - \theta_r| \leq l_0$. \end{lemma}
The following Lemma is a Proposition in the research paper of Kumar, et al.\cite{kumar2019non}. 
\begin{lemma}\cite{kumar2019non}\label{aa}
	Suppose $f(z)$ and $g(z)$ be two entire functions
	satisfying $\rho(g) < \rho(f)$. Then, for $0 < \epsilon \leq min\{\frac{3\rho(f)}{4}, \frac{\rho(f)-\rho(g)}{2}\}$, there
	exists $S \subset (1,\infty)$ with $\overline{\log dens}(S) = 1$ satisfying $$|g(z)| = o (M(|z|, f))$$
	for sufficiently large $|z| \in S$.\end{lemma}
\begin{remark}\label{remch4}
	If we replace $\rho(g)$ with $\mu(g)$, then Lemma \ref{aa} would be true.
\end{remark}
In Lemma \ref{lcritical}, consider $A(z) = v(z)e^{P(z)}$, where $v(z)$ is an entire function and $P(z)$ is a polynomial of degree $n$ satisfying $\rho(v)<\deg P$. But, in Lemma \ref{LemmaManisha}, the authors considered $\rho(v)>\deg P$ and obtained that  $|A(re^{\iota\theta})|\geq exp((1-\epsilon)\delta(P,\theta)r^n)$ for $\theta \in E^+/E$ and also for $\theta \in E^-/E,$ where $E$ is a set of linear measure $0.$
\begin{lemma}\cite{kks}\label{LemmaManisha} Let $A(z) = v(z)e^{P(z)}$ be an entire function, where $P(z)$ is a polynomial of degree $n$ and $v(z)$ satisfies the condition of Theorem \ref{Manisha}.
	Then, there exists a set $E \subset [0, 2\pi]$ of linear measure zero such that for $\epsilon > 0$ the following holds:\begin{enumerate}
		\item[(i)]
		for $\theta \in E^+\setminus E,$ there exists $R(\theta) > 1$ such that \begin{equation}\label{eqManisha1}
			| A(re^{\iota\theta})|\geq exp((1-\epsilon)\delta(P,\theta)r^n)
		\end{equation} for $r > R(\theta),$
		\item[(ii)] for $\theta \in E^- \setminus E,$ there exists $R(\theta) > 1$ such that
		\begin{equation}\label{eqManisha2}
			| A(re^{\iota\theta}) | \geq exp((1 -\epsilon)\delta(P,\theta)r^n)
	\end{equation}\end{enumerate}
	for $r > R(\theta)$.\end{lemma}
The following Lemma is proved by Gundersen\cite{gundersen1988finite}, it gives the logarithmic estimate of the analytic function $f(z).$ 
\begin{lemma}\cite{gundersen1988finite}\label{LemmaGundersen} Let $f$ be an analytic on a ray $\gamma = re^{\iota\theta}$ and suppose that for some constant $\alpha > 1,$ we have
	\begin{equation}\label{eqGundersen}
		\left|\frac{f'(z)}{f(z)}\right|=O(|z|^{-\alpha})
	\end{equation} as $z \to \infty$ along $arg z = \theta$. Then, there exists a constant $c \neq0$ such that $f(z) \to c$ as $z \to \infty$ along $arg z = \theta$.
\end{lemma}

The proof of Theorem \ref{th2} is inspired by the proof of Theorem \ref{Manisha}. We have slightly changed the proof according to the conditions of the Theorem.
\begin{proof}[\underline{Proof of Theorem \rm{\ref{th2}}}]
	If $\rho(A) = \infty,$ then it is obvious that $\rho(f) = \infty$, for all non-trivial solution $f$ of the equation (1). Therefore, let us suppose that $\rho(A) < \infty$ and there exists a non-trivial solution $f$ of the equation (1)
	such that $\rho(f) < \infty$.
	From Lemma \ref{lgundersen}, there exists $E_1 \subset [0, 2\pi]$ of linear measure zero and $m > 0$ such that,
	\begin{equation}\label{p1}
		\left|\frac{f''(re^{\iota \theta})}{f(re^{\iota \theta})}\right| \leq r^m,
	\end{equation}
	for $\theta \in [0, 2\pi] \setminus E_1$ and $r > R(\theta)$. 
	Since $A(z)$ is an entire function of finite order, suppose that $M(r,A)=|A(re^{\iota\theta_r})|$ for every $r$. Then, from Lemma \ref{fromnonhomo}, for $0<\zeta <1$ and $0<C<1$, there exists $0<l_0<\frac{1}{2}$ and $S_1\subset(0,\infty)$ with $\underline{\log dens}(S_1) \geq 1 - \zeta$ such that
	$$e^{-5\pi} M(r, A)^{1-C} \leq |A(re^{\iota \theta})|,$$ 
	for all sufficiently large $r \in S_1$ and for all $\theta$ satisfying $|\theta - \theta_r| \leq l_0$.\\
	\begin{enumerate}
		\item[(i)]Let $\rho(B)<\rho(A)$, from Lemma \ref{aa} for $0 < \epsilon \leq min\{\frac{3\rho(A)}{4}, \frac{\rho(A)-\rho(B)}{2}\}$, there
		exists $S_2 \subset (1,\infty)$ with $\overline{\log dens}(S_2) = 1$ satisfying 
		\begin{equation}\label{p2}
			\frac{|B(z)|}{M(|z|, A)} \to 0
		\end{equation}
		for sufficiently large $|z| \in S_2$.
		Using properties of logarithmic density and the fact that $\overline{\log dens}(S_1 \cup S_2) \leq 1$, we get
		\begin{align*}
			\overline{\log dens}(S_1 \cap S_2) & \geq \underline{\log dens}(S_1)+\underline{\log dens}(S_2)-\overline{\log dens}(S_1 \cup S_2)\\
			&\geq 1- \zeta +1-1=1- \zeta.
		\end{align*}
		Thus, we can choose $z_r = re^{\iota \theta_r}$ with $r \to \infty$ such that
		$r \in (S_1 \cap S_2)$ and $|A(re^{\iota\theta_r})| = M(r, A)$.
		We may consider $<\theta_r>$ as a sequence, where $r \in (S_1 \cap S_2)$ such that $\theta_{r} \to \theta_0$ and $r \in (S_1 \cap S_2)$.\\
		We may consider following three cases:-
		\begin{enumerate}
			\item[(a)]
			$\delta(P,\theta_0)>0.$\\
			From Lemma \ref{LemmaManisha}(i), we have
			\begin{equation}\label{p4}| A(re^{\iota\theta_{0}})|\geq\exp(\frac{1}{2}\delta(P,\theta_0)r),\end{equation}
			for sufficiently large $r$, where $r \in (S_1 \cap S_2)$ and $\theta_0\in E^+/E_2$ and $E_2$ is a set of critical rays of $e^{P(z)}$ of linear measure $0$.\\
			From equation (1) we get,
			\begin{equation}\label{p5}
				\left|\frac{f'(re^{\iota\theta_{0}})}{f(re^{\iota\theta_{0}})}\right|\leq\left|\frac{f''(re^{\iota\theta_{0}})}{f(re^{\iota\theta_{0}})}\right|\frac{1}{|A(re^{\iota\theta_{0}})|}+\frac{\left|B(re^{\iota\theta_{0}})\right|}{M(r, A))},
			\end{equation}
			for $r \in (S_1 \cap S_2)$ and $\theta_{0} \in E^+/(E_1 \cup E_2)$.
			Using equations (\ref{p1}), (\ref{p2}), (\ref{p4}) and (\ref{p5}), we get,
			$$\left|\frac{f'(re^{\iota\theta_{0}})}{f(re^{\iota\theta_{0}})}\right|\to0$$ for $r \in (S_1 \cap S_2)$, $r\to\infty$ and $\theta_{0} \in E^+/(E_1 \cup E_2).$
			This implies that 
			\begin{equation}\label{p6}
				\left|\frac{f'(re^{\iota\theta_0})}{f(re^{\iota\theta_{0}})}\right|=O\left(\frac{1}{r^2}\right),
			\end{equation}
			as $r \to \infty$ and $r\in S_1\cap S_2$. From Lemma \ref{LemmaGundersen},
			\begin{equation}\label{p7}
				f(re^{\iota\theta_{0}}) \to a
			\end{equation}
			as $r\to\infty$ and $r\in (S_1 \cap S_2)$
			for $\theta_{0} \in E^+ \setminus (E_1\cup E_2)$, where $a$ is a non-zero finite constant.\\
			Since $f(re^{\iota\theta _r})\to f(re^{\iota\theta_0})$ and using (\ref{p7}), we get
			$$f(re^{\iota\theta_{r}}) \to a$$ for $r\to\infty$ and $r\in (S_1 \cap S_2)$.
			Thus, entire function $f$ is bounded over domain. But since function $f$ is entire and non-constant, $f(re^{\iota\theta})$ is unbounded for all $\theta\in[0,2\pi]$. Thus, for $\theta_r\in [0,2\pi]$, function $f(re^{\iota\theta_r})$ is also unbounded, which is a contradiction.\\
			
			\item[(b)]$\delta(P,\theta)<0.$\\
			From Lemma \ref{LemmaManisha}(ii), we have
			\begin{equation}\label{p8}
				| A(re^{\iota\theta_{0}})|\geq\exp(\frac{1}{2}\delta(P,\theta_0)r^n),
			\end{equation}
			for $\theta_0\in E^-/E_1$ for large $r$.
			Using equation (\ref{p1}), (\ref{p2}) and (\ref{p8}), we have
			\begin{equation}
				\left|\frac{f'(re^{\iota\theta_{0}})}{f(re^{\iota\theta_{0}})}\right|\to0,
			\end{equation}
			as $r\to \infty$ and $\theta_0\in E^- /(E_1\cup E_2)$. From Lemma \ref{LemmaGundersen},
			\begin{equation}\label{p9}
				f(re^{\iota\theta_{0}}) \to b,
			\end{equation}
			as $r\to\infty$ and $r\in (S_1 \cap S_2)$
			for $\theta_{0} \in E^- \setminus (E_1\cup E_2)$, where $b$ is a non-zero finite constant. Since $f(re^{\iota\theta _r})\to f(re^{\iota\theta_0})$ and using (\ref{p9}), we get
			$$f(re^{\iota\theta_{r}}) \to b,$$ for $r\to\infty$ and $r\in (S_1 \cap S_2)$. Thus, entire function $f$ is bounded over whole domain. Since function $f$ is entire and non-constant, then $f(re^{\iota\theta})$ is unbounded for all $\theta\in[0,2\pi]$. Thus, for $\theta_r\in [0,2\pi]$, function $f(re^{\iota\theta_r})$ is also unbounded, which is a contradiction.\\
			\item[(c)]$\delta(P,\theta_0)=0.$\\
			Suppose  $\theta^*_0\in[0,2\pi]$ in the neighbourhood of $\theta_0$ such that $\delta(P,\theta^*_0)>0$. Letting $r\to\infty$, we get $|\theta_0 -\theta^*_0|\leq l_0$.
			Choosing $C$ and $\zeta$ such that $l_0\to0$.
			\begin{equation}\label{p10}
				\left|\frac{f'(re^{\iota\theta_{0}})}{f(re^{\iota\theta_{0}})}\right|\sim \left|\frac{f'(re^{\iota\theta^*_{0}})}{f(re^{\iota\theta^*_{0}})}\right|\leq\left|\frac{f''(re^{\iota\theta^*_{0}})}{f(re^{\iota\theta^*_{0}})}\right|\frac{1}{|A(re^{\iota\theta^*_{0}})|}+\frac{\left|B(re^{\iota\theta^*_{0}})\right|}{M(r, A))},
			\end{equation}
			Remaining proof is similar to part (i).
		\end{enumerate}
		\item[(ii)]Let $\mu(B)<\rho(A)$, from Remark \ref{remch4} for $0 < \epsilon \leq min\{\frac{3\rho(A)}{4}, \frac{\rho(A)-\mu(B)}{2}\}$, there
		exist $S_2 \subset (1,\infty)$ with $\overline{\log dens}(S_2) = 1$ satisfying 
		\begin{equation}\label{p11}
			\frac{|B(z)|}{M(|z|, A)} \to 0.
		\end{equation}
		Remaining proof is similar to part (i).
	\end{enumerate}
\end{proof}
\subsection{Second Order Non-Homogenous Linear Differential Equation}
 Kumar and Saini\cite{kumar2021growth} gave severeal results for equation \eqref{2ordernonhomo}. In one of their result, they considered $A(z)$ to have Fabry gaps, $\max ( \rho(H), \rho (B) ) < \rho (A)$ and proved the following result.  We change the condition on $A(z)$ and consider $A(z)$ to be transcendental entire function having a multiply-connected Fatou component and prove Theorem \ref{mainth1}.\\

\begin{theorem}\cite{kumar2021growth}
	Let the coefficients and H(z) of equation \eqref{2ordernonhomo} are entire functions such that $\max ( \rho(H), \rho (B) ) < \rho (A)$ and $A(z)$ has Fabry gaps. Then, any non-trivial solution of equation \eqref{2ordernonhomo} are of infinite order.
	\end{theorem}
\begin{thm}\label{mainth1}
	Let $A(z)$ be a transcendental entire function having a multiply-connected Fatou component and $B(z)$, $H(z)$ be entire functions such that $\max(\rho(H),$\linebreak $\rho(B))<\rho(A)$. Then, any non-trivial solution of \eqref{2ordernonhomo} is of infinite order.
\end{thm}
\begin{proof}[\underline{Proof of Theorem \rm{\ref{mainth1}}}]  
	Suppose $f$ is a finite order solution of equation \eqref{2ordernonhomo}. Then, applying Lemma \ref{lgundersen}, there is a set $E\subset(1,\infty)$ with finite logarithmic measure such that
	\begin{equation}\label{guneq}
		\left| \frac{f^{''}(z)}{f^{'}(z)}\right| \leq |z|^{2\rho(f)},	
	\end{equation}
	holds for all $z$ satisfying $|z|\notin E\cup [0,1]$.\\
	Given that $\max(\rho(H),\rho(B))<\rho(A)$, so let $\beta$ be such that
	$\max(\rho(H),\rho(B))<\beta<\rho(A)$, then applying the definition of order of growth on $B(z)$ and $H(z)$ gives
	\begin{equation}\label{orderproperty}
		|B(re^{\iota\theta})|\leq \exp r^{\beta}  \qquad \text{and}  \qquad |H(re^{\iota\theta})|\leq \exp r^{\beta},
	\end{equation}
	holds for all sufficiently large $r$.
	Suppose that $z_{r}=re^{\iota\theta_{r}}$ be the points such that $|f(z_{r})|=M(r,f)$. Then, applying Lemma \ref{lpantsaini}, there exists a set $F\subset(0, \infty)$ with
	$m_l(F)<\infty$ such that 
	\begin{equation}\label{impeq}
		\frac{f(re^{\iota\theta_{r}})}{f^{m}(re^{\iota\theta_{r}})}\leq 2r^{m},
	\end{equation}
	holds for all sufficiently large $r\notin F$ and for all $m\in \mathbb{N}$.
	Applying Lemma \ref{lzheng}, we have 
	\begin{equation}\label{fatoueq}
		M(r,A)^{\gamma}\leq |A(re^{\iota\theta})|,
	\end{equation}
	for $0<\gamma<1$ and $r\in F_{1}=\cup_{n=1}^\infty\{r:r_n<r<R_n\}$.\\
	From equations \eqref{2ordernonhomo}, \eqref{guneq}, \eqref{orderproperty}, \eqref{impeq}, and \eqref{fatoueq}, there exists a sequence $z=re^{\iota\theta}$ such that for all $r\in F_{1}\setminus (E\cup F\cup[0,1])$, we have
	\begin{align*}
		|A(re^{\iota\theta})|&\leq \left|\frac{f^{''}(re^{\iota\theta})}{f^{'}(re^{\iota\theta})}\right|+|B(re^{\iota\theta})|\left|\frac{f(re^{\iota\theta})}{f^{'}(re^{\iota\theta})}\right|+\left|\frac{H(re^{\iota\theta})}{f(re^{\iota\theta})}\right|\left|\frac{f(re^{\iota\theta})}{f^{'}(re^{\iota\theta})}\right|\\
		M(r,A)^{\gamma}&\leq r^{2\rho(f)}+2r\exp r^{\beta}+2r\left|\frac{H(re^{\iota\theta})}{M(r,f)}\right|\\
		&\leq r^{2\rho(f)}+4r\exp r^{\beta}\\
		&\leq 4r\exp r^{\beta}(1+o(1)).
	\end{align*}
	This gives  $\rho(A)\leq \beta$, which is a contradiction. Hence, every non-trivial solution of equation \eqref{2ordernonhomo} is of infinite order.
\end{proof}
	

\begin{thebibliography}{}
		\bibitem{banklainelangley}
		Bank, S.~B., Laine, I., and Langley, J.~K.: On the frequency of zeros of
		solutions of second order linear differential equations, Results in Math.,
		\textbf{10(1)}, 1986, 8--24.
\bibitem{frei1961losungen}
Frei, M.: {\"U}ber die l{\"o}sungen linearer differentialgleichungen mit ganzen
funktionen als koeffizienten, Comment. Math. Helv., \textbf{35}, 1961,
201--222.
		\bibitem{gao1989complex}
		Gao, S.: On the complex oscillation of solutions of non-homogeneous linear
		differential equations with polynomial coefficients, Rikkyo Daigaku sugaku
		zasshi, \textbf{38(1)}, 1989, 11--20.
		\bibitem{gundersen}
		Gundersen, G.~G.: Estimates for the logarithmic derivative of a meromorphic function, J. London Math. Soc., \textbf{2(1)}, 1988, 88--104.
		\bibitem{gundersen1988finite}
		Gundersen, G.~G.: Finite order solutions of second order linear differential
		equations, Transactions of the American Mathematical Society,
		\textbf{305(1)}, 1988, 415--429.
		\bibitem{gundersen2017research}
		Gundersen, G.: Research questions on meromorphic functions and complex
		differential equations, Computational Methods and Function Theory,
		\textbf{17(2)}, 2017, 195--209.
		\bibitem{hellerstein1991growth}
		Hellerstein, S., Miles, J., and Rossi, J.: On the growth of solutions of $ f''+
		gf'+ hf= 0$, Trans. Amer. Math. Soc., \textbf{324}, 1991, 693--706.
		\bibitem{herold1972vergleichssatz}
		Herold, H.: Ein vergleichssatz f{\"u}r komplexe lineare
		differentialgleichungen, Mathematische Zeitschrift, \textbf{126(1)}, 1972,
		91--94.
		\bibitem{kumar2019non}
		Kumar, D., Kumar, S., and Saini, M.: Non-existence of finite order solution of
		non-homogeneous second order linear differential equations, arXiv pre-prints
		arXiv:1910.03615.
		\bibitem{kks}
		Kumar, D., Kumar, S., and Saini, M.: On solution of second order complex
		differential equation, Bull. Calcutta Math. Soc., \textbf{111(4)}, 2019,
		331--340.
		\bibitem{kumarsaini}
		Kumar, S. and Saini, M.: On zeros and growth of solutions of second order
		linear differential equations, Commun. Korean Math. Soc., \textbf{35(1)},
		2020, 229--241.
		\bibitem{kumar2021growth}
		Kumar, D. and Saini, M.: The growth of solutions of non-homogeneous linear
		differential equations, Kodai Mathematical Journal, \textbf{44(3)}, 2021,
		556--574.
		\bibitem{laine1990note}
		Laine, I.: A note on the complex oscillation theory of non-homogeneous linear
		differential equations, Results in Mathematics, \textbf{18(3--4)}, 1990,
		282--285.
		\bibitem{laine1993nevanlinna}
		Laine, I.: Nevanlinna theory and complex differential equations, de Gruyter, 1993.
	\bibitem{langley1986complex}
	Langley, J.~K.: On complex oscillation and a problem of {O}zawa, Kodai Math. J., \textbf{9(3)}, 1986, 430--439.
	\bibitem{2018queslongtion}
	Long, J.~R., Shi, L., Wu, X., and Zhang, S.: On a question of {G}undersen
	concerning the growth of solutions of linear differential equations, Ann.
	Acad. Sci. Fenn. Math., \textbf{43(1)}, 2018, 337--348.
		\bibitem{marden1971logarithmic}
		Marden, M.: Logarithmic derivative of an entire function, Proceedings of the American Mathematical Society, \textbf{28(2)}, 1971, 513--518.
\bibitem{mehra2022growth}
Mehra, N., Pant, G., and Chanyal, S.~K.: Growth of solutions of complex
differential equations with entire coefficients having a multiply-connected
{F}atou component, Indian Journal of Pure and Applied Mathematics, 2022,
1--12.
		\bibitem{ozawa1980solution}
		Ozawa, M.: On a solution of $w''+ e^{-z}w'+(az+ b) w= 0$, Kodai Math. J., \textbf{3(2)}, 1980, 295--309.
		\bibitem{pant2021infinite}
		Pant, G. and Saini, M.: Infinite order solutions of second order linear differential equations, arXiv preprint arXiv:2102.11748.
		\bibitem{shi1991two}
		Shi-an, G.: Two theorems on the complex oscillation theory of non-homogeneous linear differential equations, Journal of Mathematical Analysis and
		Applications, \textbf{162(2)}, 1991, 381--391.
		\bibitem{valiron1949lectures}
		Valiron, G.: Lectures on the general theory of integral functions, Chelsea Publishing, 1949.
		\bibitem{wanglaine}
		Wang, J. and Laine, I.: Growth of solutions of second order linear differential
		equations, J. Math. Anal. Appl., \textbf{342(1)}, 2008, 39--51.
		\bibitem{wittich1966}
		Wittich, H.: Zur theorie linearer differentialgleichungen im komplexen, \textbf{379}, 1966, 363--370.
		\bibitem{zhang2018infinite}
		Zhang, G.: Infinite growth of solutions of second order complex differential equation, Open Mathematics, \textbf{16(1)}, 2018, 1233--1242.
		\bibitem{zheng}
		Zheng, J.~H.: On multiply-connected {F}atou components in iteration of meromorphic functions, J. Math. Anal. Appl., \textbf{313(1)}, 2006, 24--37.
	\end{thebibliography}
\end{document}